\documentclass[12pt,a4paper]{amsart}
\usepackage{amscd}
\usepackage{amsmath}
\usepackage{latexsym}
\usepackage{amsfonts}
\usepackage{amssymb}
\usepackage{amsthm}
\usepackage{graphicx}
\usepackage{verbatim}
\usepackage{mathrsfs}
\usepackage{enumerate}
\usepackage{color}

\newtheorem{definition}{Definition}[section]
\newtheorem{theorem}{Theorem}[section]
\newtheorem{lemma}[theorem]{Lemma}
\newtheorem{corollary}[theorem]{Corollary}
\newtheorem{proposition}[theorem]{Proposition}

\theoremstyle{remark}
\newtheorem{remark}[theorem]{Remark}

\DeclareMathOperator{\I}{|||}

\newcommand{\RR}{\mathbb{R}}

\DeclareMathOperator{\Div}{div}

\allowdisplaybreaks
\begin{document}
\title[Time-dependent singularities in the Navier-Stokes system]{Time-dependent singularities\\ in the Navier-Stokes system}

\author[G. Karch]{Grzegorz Karch}
\address[G. Karch]{Instytut Matematyczny, Uniwersytet Wroc\l awski,
 pl.~Grunwaldzki 2/4, 50-384 Wroc\l aw, Poland.}

\email{karch@math.uni.wroc.pl}

\author[X. Zheng]{Xiaoxin Zheng}

\address[X. Zheng]{Instytut Matematyczny, Uniwersytet Wroc\l awski,
pl.~Grunwaldzki 2/4, 50-384 Wroc\l aw, Poland.}

\email{xiaoxinyeah@163.com}

\date{\today}

\begin{abstract}
We show that, for a given  H\"older continuous curve in 
$\{(\gamma(t),t)\,:\, t>0\} \subset  \RR^3\times\RR^+$, there exists a  solution to the Navier-Stokes system for an incompressible fluid in $\RR^3$ which is smooth outside this curve and singular on it.
This is a pointwise solution of the system  outside  the curve, however, as a distributional solution on $\RR^3\times\RR^+$, it solves an analogous  Navier-Stokes system with a singular force concentrated on the curve.

\end{abstract}

\maketitle
\section{Introduction}\label{INTR}
\setcounter{section}{1}\setcounter{equation}{0}

The  Navier-Stokes system describing a motion of an
incompressible homogeneous fluid in the whole three dimensional space has the following form
\begin{equation}\label{eq.NS}
\partial_tu-\Delta u+(u\cdot\nabla)u+\nabla p=0,\quad \Div u=0,
\end{equation}
for $(x,t)\in\RR^3\times\RR^+$.
Here, the vector $u=\big(u_1(x,t),u_2(x,t),u_3(x,t)\big)$ denotes the unknown velocity field and the scalar function $p=p(x,t)$ stands for the unknown pressure. System \eqref{eq.NS} should be supplemented with an initial condition $u|_{t=0}=u_0$, however, it does not play any role in the statement of the main result in this work.

Our goal in this paper is to propose mathematical tools  which, in particular, allow us to prove the following theorem.
\begin{theorem}\label{THM.INT}
For every  H\"older continuous function $\gamma:\RR^+\to\RR^3$ with a H\"older exponent $\alpha>\frac34$, there exists a vector field $u=u(x,t)$ and a pressure $p(x,t)$ which  are smooth and solve system \eqref{eq.NS} for all $(x,t)\in(\RR^3\times\RR^+)\backslash\Gamma$, where $\Gamma=\{(\gamma(t),t)\in\RR^3\times\RR^+:\,t>0\}$ and which are singular on the curve $\Gamma$.
\end{theorem}
This theorem does not answer a long standing question on the existence of singular solutions to Navier-Stokes system \eqref{eq.NS}. 
In fact,  our solution $\big(u(x,t),p(x,t)\big)$ satisfies system \eqref{eq.NS} in a pointwise sense for all $(x,t)\in(\RR^3\times\RR^+)\backslash\Gamma$, however, as a distributional solution on $\RR^3\times\RR^+$, the couple $(u,p)$ solves an analogous Navier-Stokes system with an external force concentrated on the curve $\Gamma$, see equation~\eqref{eq.SNS} below.

An example of such  a solution has been already  obtained in a physical  experiment, where an axially symmetric jet discharging from a  thin pipe into the  space was studied.   This phenomenon may be described
by the classical incompressible Navier-Stokes system  \eqref{eq.NS}
 and its one-parameter family explicit stationary solutions $\big(V^c(x),Q^c(x)\big)$ of the following form
\begin{eqnarray}
V_1^c(x)= 2{c|x|^2-2x_1|x|+cx_1^2\over |x|(c|x|-x_1)^2},  &&\quad
V_2^c(x)= 2{x_2(cx_1-|x|) \over |x|(c|x|-x_1)^2}, \label{sing-sol} \\
V_3^c(x)= 2{x_3(cx_1-|x|)  \over |x|(c|x|-x_1)^2},&&\quad
Q^c(x)= 4{cx_1-|x|  \over |x|(c|x|-x_1)^2}, \nonumber
\end{eqnarray}
where $|x|=\sqrt{x_1^2+x_2^2+x_3^2}$ and $c\in\RR$ is an arbitrary constant
such that $|c|>1$.
These explicit stationary solutions to \eqref{eq.NS} seem to be  discovered
 first
 by Slezkin \cite{Slezkin}  (see the translation of this work in \cite[Apendix]{Gal}) and described by Landau in \cite{L} (see also \cite[Sec.~23]{LL}).
They play a pivotal role in this work and we are going to call them as
{\it the Slezkin-Landau solutions} to system \eqref{eq.NS}.
Let us also recall that the stationary solutions \eqref{sing-sol} were also independently derived in \cite{Sq,TX}
and they can be found in standard textbooks,
see {\it e.g.}~\cite[p.~206]{B}.
To obtain such solutions, it suffices to notice
 that the additional axisymmetry requirement reduces the stationary Navier-Stokes system to a system of ODEs which can be solved explicitly in terms of elementary functions. Recently, \v{S}ver\'ak \cite{Sv} proved that even if we drop the requirement of axisymmetry,
 the Slezkin-Landau solutions \eqref{sing-sol} are still the only stationary solutions  which are invariant under the natural scaling of system \eqref{eq.NS}.

The Slezkin-Landau solutions appear in recent works in different contexts.
It is proved in \cite{CK04} that they are asymptotically stable in a suitable
Banach space of tempered distributions. They are also  asymptotically stable
under arbitrary large initial perturbations of  finite energy, see \cite{KP11,KPS13}.
They appear in asymptotic expansions of solutions to initial-boundary value problems for the Navier-Stokes system \eqref{eq.NS},
{\it cf.}~\cite{FGK11,KMT12,AV11,MT12}.

One can check by  straightforward calculations that the functions
$\big(V_1^c(x), V_2^c(x), V_3^c(x)\big)$ and $Q^c(x)$ given by \eqref{sing-sol} satisfy system \eqref{eq.NS} in the {\it pointwise sense} for every $x\in
\RR^3\setminus\{(0,0,0)\}$.
They  are homogeneous functions of degree $-1$ and $-2$, respectively, and are smooth for $x\neq0$.
Thus, the Slezkin-Landau  solution $\big(V^c,Q^c\big)$ solves system \eqref{eq.NS} in a classical and pointwise sense on $(\RR^3\times\RR^+)\backslash\Gamma_0$ and is singular on the line $\Gamma_0=\{(0,t)\in\RR^3\times\RR^+,\,t\geq0\}$.
On the other hand, if one treats $\big(V^c,Q^c\big)$ as a {\it distributional or generalized} solution to
\eqref{eq.NS} in the whole space $\RR^3$, it corresponds to the singular external force
$(\kappa\delta_0, 0,0),$ where the parameter $\kappa\neq 0$ depends on $c$ and $\delta_0$ stands for the Dirac measure (details of this reasoning are recalled  below in Proposition \ref{prop:sing-sol}).

In this work, we generalize this idea and we construct analogous singular solutions on a sufficiently regular curve $\Gamma\subset\RR^3\times\RR^+$.
Our solutions are not explicit and they behave asymptotically as the Slezkin-Landau solution in a neighborhood of a singularity at the curve $\Gamma$.
Theorem \ref{THM.INT} is a particular case of Theorem \ref{THM.SNS} formulated in the next section.

Our results  have been motivated by recent works of Yanagida and his collaborators \cite{SY09,SY12,SY12-2,SY11,SY10,TY13}
where solutions singular on curves have been constructed 
for either nonlinear or linear heat equation.

\subsection*{Notation}
Here, $\RR^+=(0,\infty)$.
For $p \in [1,\infty]$,
 the usual  Lebesgue space is denoted by $L^p (\RR^3)$
and  the weak  Marcinkiewicz $L^p$-space --
by $L^{p,\infty}(\RR^3)$.
In the case of  Banach spaces $X$ used in this work, the norm in $X$
is denoted by $\|\cdot\|_X$.
Given an open set $\Omega$, the symbol
 $C^\infty_{\rm c} (\Omega)$ denotes the  set of all smooth functions which are compactly supported in $\Omega$.
   $\mathcal{S}(\RR^3)$ is the Schwartz class of smooth and rapidly decreasing functions.
   Constants
 may change from line to line and  will be denoted by $C$.

\section{Main result}\label{MAIN}
\setcounter{section}{2}\setcounter{equation}{0}

The result formulated in Theorem \ref{THM.INT} is covered by the following more general result.
\begin{theorem}\label{THM.SNS}
Assume that $\gamma:[0,\infty)\rightarrow\RR^3$ is H\"older continuous with an exponent $\alpha\in \big(\frac12,1\big]$ and denote by $\Gamma$ the curve  $\{(\gamma(t),t)\in\RR^3\times\RR^+:\,t>0\}$.
There exists a vector field $u(x,t)=(u_1,u_2,u_3)\in L^\infty\big([0,\infty);L^{3,\infty}(\RR^3)\big)$ and a pressure $p\in L^\infty\big([0,\infty);L^{\frac32,\infty}(\RR^3)\big)$ such that
\begin{enumerate}[\rm (i)]
\item \label{eq.thm-item-0} $u(x,t)$ and $p(x,t)$ satisfy system \eqref{eq.NS} for all $(x,t)\in(\RR^3\times\RR^+)\backslash\Gamma$ in the sense of distributions,
  \item\label{eq.thm-item-1}\begin{itemize}
                              \item $u\in L^\infty_{\rm loc}\big((\RR^3\times\RR^+)\backslash\Gamma\big)$ if $\alpha\in(\frac12,\frac34]$,
                              \item $u,\,p\in C^\infty\big((\RR^3\times\RR^+)\backslash\Gamma\big)$ if $\alpha>\frac34$.                         \end{itemize}
   \item\label{eq.thm-item-2}
for every $t>0$
\begin{itemize}
\item $u(\cdot,t)-V^c(\cdot-\gamma(t))\in L^q(\RR^3)$ for each $q\in\big(3,\frac{3}{2-2\alpha}\big)$,
\item $p(\cdot,t)-Q^c(\cdot-\gamma(t))\in L^q(\RR^3)$ for each $q\in\big(\frac32,\frac{3}{3-2\alpha}\big)$,
\end{itemize}
where $(V^c, Q^c)$ denotes the Slezkin-Landau solution given by formula \eqref{sing-sol} with fixed and sufficiently large $|c|>1$.
\end{enumerate}
\end{theorem}
\begin{remark}
We say that $(u,p)$ solves system \eqref{eq.NS} in the sense of distributions on $\{(\RR^3\times\RR^+)\backslash\Gamma\}$ if the usual
 distributional integral formulation of system \eqref{eq.NS}
 is satisfied for all test functions $\varphi\in C^\infty_{\rm c}\big((\RR^3\times\RR^+)\backslash\Gamma\big)$.
\end{remark}

\begin{remark}\label{rem:sing}
Property \eqref{eq.thm-item-2} of Theorem \ref{THM.SNS} means that $u(x,t)$ and $p(x,t)$ have to be singular at the point $x=\gamma(t)$ for every $t>0$ and this singularity is comparable with the singularity of the functions $V^c(x-\gamma(t))$ and of $Q^c(x-\gamma(t))$, respectively. This is due to the fact that $V^c=V^c(x)$ is a homogeneous function of degree $-1$,
so, it does not belong to $L^q_{\rm loc}(\RR^3)$ for each $q\geq3$. Similarly, $Q^c=Q^c(x)$ is homogeneous of degree $-2$ and it does not belong to $L^q_{\rm loc}(\RR^3)$ for each $q\geq\frac32$.
\end{remark}

\begin{remark}\label{REM.Landan}
As we have emphasized in Remark \ref{rem:sing},
the stationary solution $V^c=V^c(x)$ defined in
\eqref{sing-sol} is singular with singularity of the kind
$\mathcal{O}(1/|x|)$ as $|x|\to 0$. This is the critical singularity,
because as it was shown
in \cite{CK,KK06}, every pointwise stationary solution
to system \eqref{eq.NS} in
$B_R\setminus\{0\}=\{x\in\RR^3\,:\, 0<|x|<R\}$ satisfying
$u(x)=o(1/|x|)$ as $|x|\to 0$ is also a solution in the sense of
distributions in the whole $B_R$.
 In other words, such a singularity at the origin is removable.
Analogous results on
removable singularities of time-dependent weak solutions to
the Navier-Stokes equations has been proved by Kozono \cite{Kozone98}.
Recent results on removable (time-dependent) singularities to semilinear parabolic equations can be found in
\cite{H13,Hsu2010,TY13}.
\end{remark}

The velocity vector field and the pressure $(u,p)$ in Theorem \ref{THM.SNS} are obtained as solutions to the following initial value problem for the incompressible Navier-Stokes system with a singular force
\begin{equation}\label{eq.SNS}
\begin{array}{ll}
\partial_tu+(u\cdot\nabla)u-\Delta u+\nabla p=\kappa\delta_{\gamma(t)}\bar{e}_1,\quad(x,t)\in\RR^3\times\RR^+,\\
\Div u=0,\\
u|_{t=0}=0,
\end{array}
\end{equation}
where $\kappa\in\RR$, $\bar{e}_1=(1,0,0)$ and $\delta_{\gamma(t)}$ is the Dirac measure on $\RR^3$ concentrated at the point $x=\gamma(t)$. Solutions to  the singular problem \eqref{eq.SNS} with sufficiently small $|\kappa|$ has been constructed in \cite{CK04} in a suitable space of tempered distributions.
In this work, using the approach introduced in \cite{CK04}, we show that those solutions are, in fact, functions with all the 
properties stated in Theorem \ref{THM.SNS}.

\begin{remark}
For simplicity of the exposition,
we supplement problem \eqref{eq.SNS} with the zero initial datum, however, a completely analogous result may be proved in the case of a sufficiently
small initial datum $u|_{t=0}\in\mathcal{PM}^2$ (see the next section).
Such a nontrivial initial condition will just give us another solution of problem
 \eqref{eq.SNS} with properties stated in Theorem \ref{THM.SNS}.
\end{remark}

\begin{remark}
Applying results from the recent work \cite{KPS13}, we obtain immediately
that solutions to problem~\eqref{eq.SNS} are asymptotically stable under arbitrary large initial perturbations from $L^2(\RR^3)$.
\end{remark}
In the next section, we recall mathematical tools which allows us to study problem \eqref{eq.SNS}. In Section \ref{HEAT}, we explain our idea of proving Theorem \ref{THM.SNS} by using it in the case of the heat equation $u_t=\Delta u$. Theorem \ref{THM.SNS} is proved in Section \ref{PROF}.

\section{Estimates in spaces $\mathcal{PM}^a$}\label{PRE}
\setcounter{section}{3}\setcounter{equation}{0}
\subsection{Preliminary properties}
First, we precise our assumption on the curve $\gamma$.
\begin{definition}\label{def.Holder}
We say that $\gamma:\,[0,\infty)\to\RR^3$ is H\"older continuous with an exponent $\alpha\in(0,1]$, if for every $T>0$ there is a constant $C=C(T)>0$ such that for all $t,\,s\in[0,T]$, we have $\big\|\gamma(t)-\gamma(s)\big\|_{\RR^3}\leq C(T)|t-s|^\alpha$.
\end{definition}
Following \cite{CK04}, we introduce a family Banach spaces in which we solve singular problem \eqref{eq.SNS}. For every fixed $a\geq0$, we set
\begin{equation}\label{eq.PM-a}
\mathcal{PM}^a\equiv\bigl\{v\in{\mathcal{S}}'(\RR^3):\,\, \widehat v\in L^1_{\rm loc}(\RR^3),\,
\|v\|_{\mathcal{PM}^a}\equiv \text{ess}\sup_{\xi\in\RR^3} |\xi|^a|\widehat
v(\xi)|<\infty\bigr\}.
\end{equation}
In this paper, we mainly deal with vector fields $u=(u_1,u_2,u_3)$ hence, by the very definition
$\|u\|_{\mathcal{PM}^a}=\max\bigl\{\|u_1\|_{\mathcal{PM}^a},\|u_2\|_{\mathcal{PM}^a},\|u_3\|_{\mathcal{PM}^a}\bigr\}.$
Below, we construct a solution to problem \eqref{eq.SNS} satisfying $u\in C_{\rm w}\big([0,\infty),\mathcal{PM}^2\big)$ which means that $u\in L^\infty\big([0,\infty);\mathcal{PM}^2\big)$ and the function $\langle u(t),\varphi\rangle$ is continuous with respect to $t\geq0$ for every test function $\varphi\in\mathcal{S}(\RR^3)$. Here, $\langle\cdot,\cdot\rangle$ denotes the usual pairing between $\mathcal{S}'(\RR^3)$ and $\mathcal{S}(\RR^3)$.

To study regularity properties of $u$, also following \cite{CK04}, we introduce the Banach space
\begin{eqnarray}
\mathcal{Y}_T^a&\equiv&  C _{\rm w}([0, T), \mathcal{PM}^2)\label{Ya}\\
&&\cap\;\; \big\{v:(0,T)\to
\mathcal{PM}^a:\,\,
\I v\I_{a,T}\equiv\sup_{0<t\leq T}t^{a/2-1}\|v(t)\|_{\mathcal{PM}^a}<\infty\big\},\nonumber
\end{eqnarray}
for each $a\geq2$ and $T\in(0,\infty]$.
The space $\mathcal{Y}_T^a$ is normed by the quantity
$\|v\|_{\mathcal{Y}_T^a}=\I v\I_{2,T}+\I v\I_{a,T}$ and of course, $\mathcal{Y}^{2}_{\infty}=C_{\rm w}\big([0,\infty),\mathcal{PM}^2\big)$ with this definition.

This is a usual procedure to eliminate the pressure $p=p(x,t)$ from problem \eqref{eq.SNS} by applying the Leray projector $\mathbb{P}$ on solenoidal vector fields, which is given formally by the formula
$\mathbb{P}={\rm I}-\nabla(\Delta)^{-1}\Div.$
~It is well-known that this is a pseudodifferential operator corresponding to the matrix
$\widehat{\mathbb{P}}(\xi)$  with the components
\begin{equation}\label{eq.Leray-pro}
\big(\widehat{\mathbb{P}}(\xi)\big)_{j,k} =\delta_{jk} -{\xi_j\xi_k\over |\xi|^2},
\end{equation}
where $\delta_{jk}=0$ for $j\neq k$ and $\delta_{jk}=1$ for $j=k$.
In particular, $|\widehat{\mathbb{P}}(\xi)|\leq2$ for all $\xi\neq0$. Using the Leray projector $\mathbb{P}$, at least formally, we may rewrite problem \eqref{eq.SNS} as the following integral equation
\begin{equation}\label{eq.Integral}
u(t)=B(u,u)(t)+\kappa\int_0^tS(t-\tau)\mathbb{P}(\delta_0\bar{e}_1)\,\mathrm{d}\tau,
\end{equation}
where $S(t)$ is the heat semigroup given as the convolution with the
Gauss--Weierstrass kernel  $G(x,t)=(4\pi t)^{-3/2} \exp(-|x|^2/(4t))$ and the bilinear form
\begin{equation}\label{Bf}
B(u,v)(t)=-\int_0^t S(t-\tau)\mathbb{P} \nabla \cdot (u\otimes
v)(\tau)\;\mathrm{d}\tau,
\end{equation}
with $u\otimes v=(u_iv_j)_{3\times3}$. To give a precise meaning of equations \eqref{eq.Integral}-\eqref{Bf}, we reformulate them using the Fourier transform in the following way.

\begin{definition}\label{def1}
By a solution of either  problem \eqref{eq.SNS} or equations \eqref{eq.Integral}-\eqref{Bf} we mean a vector field
$u= (u_1, u_2,u_3)\in C_{\rm w}\big([0,\infty);\mathcal{PM}^2\big)$ such that
\begin{equation}\label{FDuh}
\widehat u(\xi,t) =\int_0^te^{-(t-\tau)|\xi|^2}
\widehat{\mathbb{P}} (\xi)\; i\xi\cdot \big(\widehat{u\otimes
u}\big)(\xi,\tau)\, \mathrm{d}\tau
+\kappa\int_0^te^{-(t-\tau)|\xi|^2}
\widehat{\mathbb{P} }(\xi) e^{i\gamma(\tau)\cdot\xi}\bar{e}_1\, \mathrm{d}\tau
\end{equation}
    for all $t\geq0$ and almost all $\xi\in\RR^3$, where for two tempered distributions  $u,\,v\in\big(\mathcal{PM}^2\big)^3$, we denote $\widehat{u\otimes
v}=(\widehat{u_i}\ast\widehat{v_j})_{3\times3}$.
\end{definition}

We construct a solution to integral equation \eqref{FDuh}  using the Banach fixed point theorem which, in the case of the incompressible Navier-Stokes system, is often reformulated in the following way.
\begin{lemma}\label{lem:xyB}
Let $(X,\|\cdot\|)$ be an abstract Banach space, $L:\,X\to X$ be a linear bounded operator such that for a constant $\lambda\in[0,1)$, we have
$\|L(x)\|\leq \lambda\|x\|$ for all $x\in X,$
and $B :X \times X \to X$ be a bilinear mapping such that
$$\|B(x_1,x_2)\|\le \eta\|x_1\|\,\|x_2\|\quad\text{for every}\quad x_1,\, x_2\in X $$
for some constant $\eta>0$. Then, for every $y\in X$ satisfying $4\eta\|y\|< (1-\lambda)^{2}$, the equation
\begin{equation}\label{eq.fixed-theorem}
x = y +L(x)+B(x,x)
\end{equation}
 has a solution $x\in X$. In particular, this solution satisfies $\|x\|\le \frac{2\|y\|}{1-\lambda}$,
and it is the only one among all solutions satisfying $\|x\|< \frac{1-\lambda}{2\eta}$.
\end{lemma}
We skip the proof of this lemma because it is elementary, and  it simply 
consists in applying the Banach contraction principle to equation \eqref{eq.fixed-theorem} in the ball $\{x\in X:\,\|x\|\leq\varepsilon\}$ with arbitrary $\varepsilon<\frac{1-\lambda}{2\eta}$.
\subsection{Auxiliary estimates}
Due to Lemma \ref{lem:xyB}, to show the existence of solutions to equation \eqref{FDuh}, we need estimates of the bilinear form $B(\cdot,\cdot)$ defined in \eqref{Bf}.
\begin{lemma}\label{lem.bil}
Let $2\leq a<3$ and $T\in(0,\infty]$. There exists a constant $\eta_a>0$, independent of $T$, such that
for every $u\in C _{\rm w} \big([0,T), \mathcal{PM}^2\big)$ and
$v\in \bigl\{v(t)\in\mathcal{PM}^a:\; \I v\I_{a,T}<\infty\bigr\}$ we have
\begin{equation}\label{eq.Ya-Est}
\I B(u,v)\I_{a,T}\leq \eta_a \I u\I_{2,T} \I v\I_{a,T},
\end{equation}
where the norm $\I\cdot\I_{a,T}$ is defined in \eqref{Ya}.
\end{lemma}
We skip the proof of Lemma \ref{lem.bil}, because it was proved in \cite[Proposition 7.1]{CK04}. Notice that in \cite{CK04}, inequality \eqref{eq.Ya-Est} was shown for $T=\infty$, however, after a minor modification that proof works for each finite $T$ as well.

\begin{lemma}\label{lem.product}
Let $a,\,b\in (0,3)$ and $a+b<3$.
There exists a constant $C>0$ such for all  $u\in\mathcal{PM}^a$ and $v\in\mathcal{PM}^b$
\begin{equation*}
\|uv\|_{\mathcal{PM}^{3-(a+b)}}\leq C\|u\|_{\mathcal{PM}^a}\|v\|_{\mathcal{PM}^b}.
\end{equation*}
\begin{proof}
This inequality is an immediate consequence of the following estimate
\begin{equation*}
\begin{split}
\widehat{(uv)}(\xi)=&C\int_{\RR^3}\widehat{u}(\xi-\eta)\widehat{v}(\eta)\,\mathrm{d}\eta\\
\leq&C\|u\|_{\mathcal{PM}^a}\|v\|_{\mathcal{PM}^b}\int_{\RR^3}\frac{1}{|\xi-\eta|^a}\frac{1}{|\eta|^b}\,\mathrm{d}\eta= \frac{C}{|\xi|^{3-(a+b)}}\|u\|_{\mathcal{PM}^a}\|v\|_{\mathcal{PM}^b}
\end{split}
\end{equation*}
for all $\xi\in \RR^3\setminus \{0\}$.
\end{proof}
\end{lemma}

To study a regularity of solutions to equation \eqref{FDuh}, we need
some imbedding properties of spaces $\mathcal{PM}^a$. We formulate then in the following three lemmas.
\begin{lemma}\label{lem:interpol}
Let $0\leq a<b<3$ and $b>\frac{3}{2}$.
Then there
exists a positive constant $C=C(a,b,q)$  such that
for all $v\in \mathcal{PM}^a\cap\mathcal{PM}^b$ we have
\begin{equation}
\|v\|_{L^q(\RR^3)} \leq C\|v\|_{\mathcal{PM}^a}^{1-\beta}
\|v\|_{\mathcal{PM}^b}^\beta\quad\text{for all}\quad q\in I=\begin{cases}[2,\frac{3}{3-b}), \quad&\text{if}\,\,\,\,0\leq a<\frac32\\
(\frac{3}{3-a},\frac{3}{3-b}), \quad&\text{if}\,\,\,\,\frac32\leq a<3
\end{cases},
\label{interpol}
\end{equation}
where $\beta= \frac{q(3-a)-3}{q(b-a)}$.
\end{lemma}
\begin{proof}
In the particular case of $a=2$, this Lemma was proved in \cite[Lemma 7.4]{CK04}. For the completeness of the exposition, we show the general case.
Using a standard approximation procedure one may assume that $v$ is smooth and rapidly decreasing. Since $q\geq2$, by the
Hausdorff--Young inequality  with $1/p+1/q=1$ and $p\in [1,2]$ and
the definition of the $\mathcal{PM}^a$-norm we obtain
\begin{equation}
\begin{split}
\|v\|_q^p\leq& C\|\widehat v \|_p^p \leq C\|v\|_{\mathcal{PM}^a}^p \int_{|\xi|\leq
R}
{1\over |\xi|^{ap}}\;\mathrm{d}\xi +C\|v\|_{\mathcal{PM}^b}^p \int_{|\xi|> R}
{1\over |\xi|^{bp}}\;\mathrm{d}\xi\nonumber\\
\leq& C\|v\|_{\mathcal{PM}^a}^pR^{3-ap}+C\|v\|_{\mathcal{PM}^b}^pR^{3-bp}\label{interpol2}
\end{split}
\end{equation}
for all $R>0$ and $C$ independent of $v$ and $R$. In these
calculations, we require $ap<3$ which is equivalent to $q>3/(3-a)$.
Moreover, we have to assume that $bp>3$ which leads to the inequality
$q<3/(3-b)$. Now, we optimize inequality \eqref{interpol2} with respect
to $R$ to get formula \eqref{interpol}.
\end{proof}
\begin{lemma}\label{lem.imbed}
Let  $\beta>\frac{5}{2}$. Then there exist a constant $C>0$ such that
\begin{equation}\label{eq.imbed}
\|\nabla\omega\|_{L^2(\RR^3)}\leq C\|\omega\|^{1-\theta}_{\mathcal{PM}^2}\|\omega\|^{\theta}_{\mathcal{PM}^\beta}\quad\text{for all}\quad \omega\in\mathcal{PM}^2\cap\mathcal{PM}^\beta,
\end{equation}
where $\theta=\frac{2\beta-5}{2(\beta-2)}$.
\end{lemma}
\begin{proof}
According to the Plancherel theorem, we have
$\|\nabla\omega\|_{L^2(\RR^3)}^2= \int_{\RR^3}|\xi|^2|\widehat{\omega}(\xi)|^2\,\mathrm{d}\xi$.
As in the proof of Lemma \ref{lem:interpol}, we decompose this integral with respect to $\xi$ into two parts $\int_{|\xi|\leq R}|\xi|^2|\widehat{\omega}(\xi)|^2\,\mathrm{d}\xi$ and
$\int_{|\xi|>R}|\xi|^2|\widehat{\omega}(\xi)|^2\,\mathrm{d}\xi$, where $R$ is a positive number to be fixed later. Next, we deal with the each term separately.

In the case of the integral over the low frequency domain $|\xi|\leq R$, we estimate  as follows
\begin{equation*}
\int_{|\xi|\leq R}|\xi|^2|\widehat{\omega}(\xi)|^2\,\mathrm{d}\xi\leq \|\omega\|_{\mathcal{PM}^2}^2\int_{|\xi|\leq R}|\xi|^{-2}\,\mathrm{d}\xi\leq C N \|\omega\|_{\mathcal{PM}^2}^2.
\end{equation*}
For high frequencies $|\xi|>R$, by the definition of space $\mathcal{PM}^\beta$, we have
\begin{equation*}
\int_{|\xi|>R}|\xi|^2|\widehat{\omega}(\xi)|^2\,\mathrm{d}\xi\leq \|\omega\|_{\mathcal{PM}^\beta}^2\int_{|\xi|> R}|\xi|^{-2\beta+2}\,\mathrm{d}\xi\leq C N^{-2\beta+5} \|\omega\|_{\mathcal{PM}^\beta}^2.
\end{equation*}
Choosing  $R^{\beta-2}= \frac{\|\omega\|_{\mathcal{PM}^\beta}}{\|\omega\|_{\mathcal{PM}^2}},$
we obtain estimate \eqref{eq.imbed}.
\end{proof}

Next, we show that tempered distributions from $\mathcal{PM}^2$ are in fact functions from $L^{3,\infty}(\RR^3).$
\begin{lemma}\label{lem.weakl-3}
There exists a constant $C$ such that
\begin{equation}\label{eq.weakl-3}
\|u\|_{L^{3,\infty}(\RR^3)}\leq C\|u\|_{\mathcal{PM}^2}\quad\text{for all}\quad u\in\mathcal{PM}^2.
\end{equation}
\end{lemma}
\begin{proof}
The imbedding $\mathcal{PM}^2\subset L^{3,\infty}(\RR^3)$ has been noticed in \cite[Remark 3.2]{KPS13}. Here, we present its proof for the completeness of the exposition. First, we recall (see for example \cite{ONeil63}) the Hausdorff--Young inequality in the Lorentz space
\begin{equation}\label{eq.hausdorff-lor}
\|\widehat{f}\|_{L^{p,1}(\RR^3)}\leq\|f\|_{L^{p',1}(\RR^3)}
\end{equation}
for $p\geq2$ and $p'=\frac{p}{p-1}$ as well as the H\"older inequality in the Lorentz space
\begin{equation}\label{eq.holder-lor}
\|f\cdot g\|_{L^1(\RR^3)}\leq C\|f\|_{p,\infty(\RR^3)}\|g\|_{L^{p',1}(\RR^3)}.
\end{equation}
Thus, by the Plancherel theorem, the definition of the norm in $\mathcal{PM}^2$, and inequalities \eqref{eq.hausdorff-lor}-\eqref{eq.holder-lor}, for every test functions $\varphi\in\mathcal{S}(\RR^3)$, we obtain
\begin{align*}
\big|\langle u,\varphi\rangle\big|=&\Big|\int_{\RR^3}\widehat{u}(\xi)\widehat{\varphi}(\xi)\,\mathrm{d}\xi\Big|
\leq\|u\|_{\mathcal{PM}^2}\int_{\RR^3}\frac{1}{|\xi|^2}\widehat{\varphi}(\xi)\,\mathrm{d}\xi\\
\leq&\|u\|_{\mathcal{PM}^2}\big\||\cdot|^{-2}\big\|_{L^{\frac32,\infty}(\RR^3)}\|\widehat{\varphi}\|_{L^{3,1}(\RR^3)}
\leq  C\|u\|_{\mathcal{PM}^2}\|\varphi\|_{L^{\frac32,1}(\RR^3)}.
\end{align*}
Since $\mathcal{S}(\RR^3)\subset L^{3,1}(\RR^3)$ is a dense subset, we  obtain that the tempered distribution $u\in\mathcal{PM}^2$ defines a bounded linear functional on $L^{3,1}(\RR^3)$ with the norm estimated by $C\|u\|_{\mathcal{PM}^2}$. Since $\big(L^{\frac32,1}(\RR^3)\big)^*=L^{3,\infty}(\RR^3)$, we obtain immediately inequality \eqref{eq.weakl-3}.
\end{proof}

Finally, we prove a simple inequality in weak $L^p$-spaces.
\begin{lemma}\label{lem.bil-weakl3}
There exists a constant $C>0$ such that for every $u\in \big(L^{p,\infty}(\RR^3)\big)^3$ with $p\geq 2$, we have
$
\|u\otimes u\|_{L^{\frac{p}{2},\infty}(\RR^3)}\leq C\|u\|^2_{L^{p,\infty}(\RR^3)}.
$
\end{lemma}
\begin{proof}
We use the well-known fact that the norm in the Marcinkiewicz space $L^{p,\infty}(\RR^3)$ is comparable with the quantity $\sup_{\lambda\geq0}\lambda\big|\{x\in\RR^3:\,|u(x)|\geq\lambda\}\big|^{\frac1p}, $ where $|A|$ denotes the Lebesgue measure of a set $A\subset\RR^3$. Thus, by a direct calculation, we obtain the following inequalities
\begin{align*}
\sup_{\lambda\geq0}\lambda\big|\big\{x\in\RR^3:\big|(u\otimes u)(x)\big|\geq\lambda\big\}\big|^{\frac{2}{p}}
\leq&\sup_{\lambda\geq0}\lambda\big|\{x\in\RR^3:|u(x)|\geq\sqrt{\lambda}\}\big|^{\frac{2}{p}}\\
=&\left(\sup_{\lambda\geq0}\sqrt{\lambda}\big|\{x\in\RR^3:|u(x)|\geq\sqrt{\lambda}\}\big|^{\frac1p}\right)^2 \\
=&\left(\sup_{\lambda\geq0} \lambda \big|\{x\in\RR^3:|u(x)|\geq \lambda \}\big|^{\frac1p}\right)^2 \\
\leq&C\|u\|^2_{L^{p,\infty}(\RR^3)}
\end{align*}
which complete the proof of Lemma \ref{lem.bil-weakl3}.
\end{proof}

\subsection{Slezkin-Landau solution}
To conclude this section, we recall properties of the Slezkin-Landau solution given by formula \eqref{sing-sol}.
\begin{proposition}\label{prop:sing-sol}
Let $V^c=(V_1^c,V_2^c,V_3^c)$ and $Q^c$ be defined by \eqref{sing-sol}. For every
test function $\varphi \in \big(C^\infty_{\rm c} (\RR^3)\big)^3$ the following
equalities hold true:
\begin{equation}
\int_{\RR^3} V^c\cdot \nabla \varphi\, \mathrm{d}x =0, \label{w-div}
\end{equation}
and
\begin{equation}
\int_{\RR^3} \left(\nabla V_k^c\cdot \nabla \varphi -
V^c_kV^c\cdot \nabla \varphi -Q^c {\partial\over \partial x_k} \varphi\right) \mathrm{d}x=
\left\{
\begin{array}{lcl}
\kappa(c)\varphi(0) &\mbox{if} & k=1\\
0&\mbox{if} & k=2,3,
\end{array}
\right.\label{w-eq}
\end{equation}
where
\begin{equation}\kappa(c) =
{8\pi c\over 3(c^2-1)} \left(2+6c^2 -3c(c^2-1)\log\left({c+1\over
c-1}\right) \right).\label{fact.b}
\end{equation}
In particular, the function $\kappa=\kappa(c)$ is decreasing on $(-\infty, -1)$
and $(1,+\infty)$.
Moreover, $\lim_{c \searrow 1} \kappa(c)=\infty$, $\lim_{c \nearrow -1}
\kappa(c)=-\infty$ and
$\lim_{\vert c\vert \to\infty } \kappa(c)=0$.
\end{proposition}
The detailed proof of Proposition \ref{prop:sing-sol} is given in \cite[Proposition 2.1]{CK04}.

We conclude this section by showing that the Slezkin-Landau solution is small in the sense of the $\mathcal{PM}^2$-norm for large $|c|$.
\begin{lemma}\label{lem.Landau-small}
Let $V^c(x)$ be the Slezkin-Landau velocity field given by formula \eqref{sing-sol}. Then, there exists a constant $K>0$ independent of $c$ such that
\begin{equation}\label{eq.Landau-small}
\|V^{c}\|_{\mathcal{PM}^2}\leq \frac{K}{|c|}\quad\text{for all}\quad |c|>2.
\end{equation}
\end{lemma}
\begin{proof}
It follows from the explicit formula for $V^c$ that
\begin{equation*}
V^{c}_1=\frac{2}{c|x|}\left(\frac{1}{\big(1-c\frac{x_1}{|x|}\big)^2}-\frac{\frac{x_1}{|x|}}{\big(1-c\frac{x_1}{|x|}\big)^2}
+\frac{\frac{x_1^2}{|x|^2}}{\big(1-c\frac{x_1}{|x|}\big)^2}\right),
\end{equation*}
where the functions $\frac{1}{(1-c\frac{x_1}{|x|})^2}$, $\frac{{x_1}/{|x|}}{(1-c\frac{x_1}{|x|})^2}$, $\frac{{x_1^2}/{|x|^2}}{(1-c\frac{x_1}{|x|})^2}$ and their derivatives are bounded on $\RR^3$ uniformly in $|c|\geq2$. A similar reasoning should be applied in the case of $V^c_2(x)$ and $V^c_3(x)$. Thus, by direct calculations, for every multiindex $\alpha=(\alpha_1,\alpha_2,\alpha_3)$ and $D^\alpha=\partial_{x_1}^{\alpha_1}\partial_{x_2}^{\alpha_2}\partial_{x_3}^{\alpha_3}$,
 we have
\begin{equation}\label{eq.Landanu-der}
 |D^\alpha V^{c}(x)|\leq\frac{ C(\alpha)}{|c|} \cdot\frac{1}{|x|^{1+|\alpha|}}\quad\text{for all}\quad x\in\RR^3\backslash\{0\},
 \end{equation}
where constants $C(\alpha)$ are independent of $c\in\RR$ such that $|c|\geq2$.

Next, using the Littlewood--Paley theory, one can write the following decomposition for all $\xi\neq0$
 \begin{equation*}
\widehat{V}^c(\xi)=\sum_{q\leq\frac{1}{|\xi|}}\dot{\Delta}_q\widehat{V}^c(\xi)+\sum_{q>\frac{1}{|\xi|}}\dot{\Delta}_q\widehat{V}^c(\xi)=I+II,
 \end{equation*}
 where  for  each $q\in\mathbb{Z}$ the symbol $\dot{\Delta}_q=\varphi(2^{-q}D)$ denotes the  {\it homogeneous Littlewood--Paley operator} with $\varphi\in C_{\rm c}^\infty(\RR^3)$  supported in the annulus $\big\{x\in\mathbb{R}^{3}\,:\,\frac{3}{4}\leq|x|\leq\frac{8}{3}\big\}$ and satisfying
$
\sum_{q\in \mathbb{Z}}\varphi(2^{-q}x)=1
$
for  each
$
x\in \mathbb{R}^{3}\backslash\{0\}.
$
We refer the reader to \cite{Cannone} and to references therein
for more results on
 the Littlewood--Paley decomposition. Now, we deal with the terms $I$ and $II$, separately.

In the case where $q\leq\frac{1}{|\xi|}$, by the Hausdorff--Young inequality, the support property of $\varphi$, and by estimate \eqref{eq.Landanu-der} with $\alpha=(0,0,0)$, term $I$ can be bounded as follows
 \begin{align*}
 I\leq C \sum_{q\leq\frac{1}{|\xi|}}\int_{\RR^3}\varphi(2^{-q}x)|V^c(x)|\,\mathrm{d}x
 \leq&\frac{C(\alpha)}{|c|}2^{-q}\int_{\RR^3}\varphi(2^{-q}x)\,\mathrm{d}x\\
 \leq&\frac{C}{|c|}\sum_{q\leq\frac{1}{|\xi|}}2^{2q}\|\varphi\|_{L^1(\RR^3)}
 \leq \frac{C}{|c|} \frac{1}{|\xi|^2}.
 \end{align*}

Next, for $q>\frac{1}{|\xi|}$, the term $II$ can be written as $|\xi|^{-|\alpha|}\sum_{q>\frac{1}{|\xi|}}|\xi|^{|\alpha|}\dot{\Delta}_q\widehat{V}^c(\xi)$  for each $|\alpha|\geq0$.
Moreover, by the Hausdorff--Young inequality and the Leibniz formula, we easily find that
\begin{equation}\label{eq.lsmall-1}
\begin{split}
|\xi|^{|\alpha|}\dot{\Delta}_q\widehat{V}^c(\xi)\leq&\int_{\RR^3}\big|D^\alpha\big(\varphi(2^{-q}x)V^c(x)\big)\big|\,\mathrm{d}x\\
\leq&C\sum_{|\alpha_1|+|\alpha_2|=|\alpha|}\int_{\RR^3}\big|D^{\alpha_1}\big(\varphi(2^{-q}x)\big)\big(D^{\alpha_2}V^c(x)\big)\big|\,\mathrm{d}x.
\end{split}
\end{equation}
Now, we fix $\alpha$ such that $|\alpha|=3$. Inequality \eqref{eq.lsmall-1} together with the support property of $\varphi$ and estimate \eqref{eq.Landanu-der} allows us to conclude that
 \begin{align*}
II\leq&\frac{C}{|c|} |\xi|^{-3}\sum_{q>\frac{1}{|\xi|}}\sum_{|\alpha_1|+|\alpha_2|=|\alpha|}2^{-4q}\int_{\RR^3}\big|(D^{\alpha_1}\varphi)(2^{-q} x)\big|\,\mathrm{d}x\nonumber
\\\leq&\frac{C}{|c|} |\xi|^{-3}\sum_{q>\frac{1}{|\xi|}}2^{-q}\leq \frac{C}{|c|} \frac{1}{|\xi|^2}.
\end{align*}
Both estimates of $I$ and $II$ yields the required inequality \eqref{eq.Landau-small}.
\end{proof}

\section{Singular solutions to the heat equation}\label{HEAT}
\setcounter{section}{4}\setcounter{equation}{0}

 In order to illustrate our method of constructing singular solutions to the Navier-Stokes system \eqref{eq.NS}, we apply it first to the linear heat equation
 \begin{equation}\label{eq.heat}
 \partial_tu-\Delta u=0,\quad (x,t)\in\RR^3\times\RR^+.
 \end{equation}
 We are going to construct a pointwise solution to \eqref{eq.heat} which is singular on an arbitrary H\"older continuous curve $\Gamma=\{(\gamma(t),t):\, t>0 \,\,\text{and}\,\,\gamma(t):\RR^+\to\RR^3\}$.

 First, we recall that the function
$
U(x)=\frac{1}{4\pi|x|} $
satisfies the Poisson equation
\begin{equation*}
-\Delta U(x)=0 \quad \text{for all}\quad x\in \RR^3\backslash\{0\}.
\end{equation*}
Thus, $U=U(x)$ is a one point singular stationary solution of the heat equation \eqref{eq.heat} with a singularity on the line
$\Gamma_0=\{(0,t)\in\RR^3\times\RR^+:\,t\geq0\}.$
However, as a distributional solution, we have
\begin{equation*}
-\int_{\RR^3}u(x)\Delta \varphi(x)\,\mathrm{d}x=\varphi(0) \quad\text{for all}\quad\varphi\in C^\infty_{\rm c}(\RR^3).
\end{equation*}
Hence, $-\Delta U=\delta_0$ in $\RR^3$, where $\delta_0$ denotes the Dirac measure.
Following this idea, we construct a solution to the inhomogeneous heat equation
\begin{equation}\label{eq.SHEAT}
\partial_tu-\Delta u=\delta_{\gamma(t)},\quad(x,t)\in\RR^3\times\RR^+,
\end{equation}
with the singular force $\delta_{\gamma(t)}$ for every $t>0$.
\begin{theorem}\label{THM.HEAT}
Let $\Gamma=\big\{(\gamma(t),t),\,t>0\big\}\subset \RR^3\times\RR^+$ be a curve, where $\gamma=\gamma(t)$ is H\"older continuous of exponent $\alpha\in(\frac12,1]$. There exists a function $u(x,t)$ such that \begin{enumerate}[\rm(i)]
                                                               \item\label{eq.item-1} $u\in C^\infty\big((\RR^3\times\RR^+)\backslash\Gamma\big)$;
                                                               \item \label{eq.item-2} $\partial_tu=\Delta u$, for all $(x,t)\in(\RR^3\times\RR^+)\backslash\Gamma$;
                                                               \item \label{eq.item-3} we have the decomposition
                                                            $u(x,t)=\omega_0(x,t)+\frac{1}{4\pi|x-\gamma(t)|} $, where the function $\omega_0$  satisfies $\|\omega_0(\cdot,t)\|_{L^q(\RR^3)}
                                                            \leq Ct^{\frac12\left(\frac3q-1\right)}$ for every $q\in(3,\frac{3}{2-2\alpha})$ and $t\in(0,T]$ with a constant $C=C(q,\alpha,T)>0$.
                                                             \end{enumerate}

\end{theorem}
\begin{remark}
This theorem has been proved recently by Takahashi and Yanagida  \cite[Theorem 1.5]{TY13}. Below, we propose a different proof of Theorem \ref{THM.HEAT}, which essentially 
uses properties of the Fourier transform.

First, note  that computing formally the Fourier transform with respect to $x$ of equation~\eqref{eq.SHEAT},
we obtain the differential equation
\begin{equation}\label{eq.FHS}
\partial_t\widehat{u}(\xi,t)+|\xi|^2\widehat{u}(\xi,t)=e^{i\gamma(t)\cdot\xi} \quad\text{for all}\quad (\xi,t)\in\RR^3\times\RR^+.
\end{equation}
Supplementing this equation with the initial condition $\widehat{u}(\xi,0)=0$ for all $\xi\in\RR^3$, we obtain the following formula for its solution
\begin{equation}\label{eq.Fu}
\widehat{u}(\xi,t)=\int_0^te^{-(t-\tau)|\xi|^2}e^{i\gamma(\tau)\xi}\,\mathrm{d}\tau \quad\text{for all}\quad (\xi,t)\in\RR^3\times\RR^+.
\end{equation}
\end{remark}

\begin{proof}[Proof of Theorem \ref{THM.HEAT}]

The proof consists in showing that the tempered  distribution $u(t)$ defined in the Fourier variables by formula \eqref{eq.Fu} is, in fact, a function $u=u(x,t)$ with the properties stated in  \eqref{eq.item-1}-\eqref{eq.item-3} of Theorem \ref{THM.HEAT}.

\indent{\it Step 1.} First, let us show that $u\in C_{\rm w}\big([0,\infty);\mathcal{PM}^2\big)$; see the definition of this space in Section \ref{PRE}.
Indeed, we notice that
\begin{equation}\label{eq.u:est2}
\big|\widehat{u}(\xi,t)\big|\leq\int_0^te^{-(t-\tau)|\xi|^2}\,\mathrm{d}\tau=\frac{1}{|\xi|^2}\left(1-e^{-t|\xi|^2}\right).
\end{equation}
Hence, by the definition of the norm in $\mathcal{PM}^2$, we have
\begin{equation*}
\sup_{t\geq0}\|u(t)\|_{\mathcal{PM}^2}\leq \sup_{t>0}\sup_{\xi\in\RR^3}\left(1-e^{-t|\xi|^2}\right)\leq 1.
\end{equation*}
Now, for every test function $\varphi\in \mathcal{S}(\RR^3)$, by the definition of the Fourier transform of a tempered distribution, we have
\begin{equation}\label{eq.*}
\langle u(t),\varphi\rangle=\int_{\RR^3}\widehat{u}(\xi,t)\widehat{\varphi}(\xi)\,\mathrm{d}\xi.
\end{equation}
By formula \eqref{eq.Fu}, $\widehat{u}(\xi,t)$ is a continuous function of $t\geq0$ for each fixed $\xi\in\RR^3$. Moreover, by \eqref{eq.u:est2}, we have the inequality
$
\big|\widehat{u}(\xi,t)\widehat{\varphi}(\xi)\big|\leq\frac{1}{|\xi|^2}\big|\widehat{\varphi}(\xi)\big|,
$
where the right hand side is integrable over $\RR^3$. Thus, the continuity of the right-hand side of identity \eqref{eq.*} with respect to $t\geq0$ is an immediate consequence of the Lebesgue dominated convergence theorem.

\indent{\it Step 2.} We recall that $U(x)=\frac{1}{4\pi|x|}$ satisfies $\widehat{U}(\xi)=\frac{1}{|\xi|^2}$ for all $\xi\in\RR^3\backslash\{0\}$. Hence, for $U_{\gamma}(x,t)\equiv U(x-\gamma(t))$, we have
\begin{equation}\label{eq.FUG}
\widehat{U_\gamma}(\xi,t)=\frac{1}{|\xi|^2}e^{i\gamma(t)\cdot\xi}=\widehat{U}(\xi)e^{i\gamma(t)\cdot\xi}.
\end{equation}
We define the function
\begin{equation*}
\widehat{\omega_0}(\xi,t)=\widehat{u}(\xi,t)-\widehat{U_\gamma}(\xi,t),
\end{equation*}
which by equations \eqref{eq.Fu} and \eqref{eq.FUG} satisfies
\begin{equation}\label{eq.omega}
\begin{split}
\widehat{\omega_0}(\xi,t)=&\int_0^te^{-(t-\tau)|\xi|^2}e^{i\gamma(\tau)\cdot\xi}\;\mathrm{d}\tau-e^{i\gamma(t)\xi}\frac{1}{|\xi|^2}\\
=&-e^{-t|\xi|^2}e^{i\gamma(t)\xi}\frac{1}{|\xi|^2}+\int_0^te^{-(t-\tau)|\xi|^2}\left(e^{i\gamma(\tau)\cdot\xi}-e^{i\gamma(t)\cdot\xi}\right)\,\mathrm{d}\tau,
\end{split}
\end{equation}
because $\int_0^te^{-(t-\tau)|\xi|^2}\,\mathrm{d}\tau=|\xi|^{-2}\big(1-e^{-t|\xi|^2}\big)$.

Notice now that by \eqref{eq.u:est2} and \eqref{eq.FUG}, we have
\begin{equation}\label{eq.om:2}
\big|\widehat{\omega_0}(\xi,t)\big|\leq \frac{2}{|\xi|^2}\quad\text{for all}\quad \xi\in\RR^3\backslash\{0\},
\end{equation}
which implies that $\sup_{t>0}\|\omega_0(t)\|_{\mathcal{PM}^2}\leq2$. Moreover, repeating the argument from Step~1, we may show that $\omega_0\in C_{\rm w}\big([0,\infty);\mathcal{PM}^2\big)$.

\indent{\it Step 3.}
Now, we prove that
$\omega_0\in \mathcal{Y}^a_T$ for each $a\in [2,1+2\alpha]$ and $T>0$, where the Banach space $\mathcal{Y}_{T}^a$ is defined in \eqref{Ya}.
By Step 2, we have $\omega_0\in C_{\rm w}\big([0,\infty);\mathcal{PM}^2\big)$. Hence, it remains to estimate $\omega_0(t)$ in the $\mathcal{PM}^a$-norm, and we use here the representation of $\omega_0$ by the right-hand side of equation \eqref{eq.omega}.

First, we notice that for each $a\geq2$ there exists a positive constant $C=C(a)$ such that
\begin{equation*}
\text{ess}\sup_{\xi\in\RR^3}\left||\xi|^ae^{-t|\xi|^2}e^{i\gamma(t)\cdot\xi}\frac{1}{|\xi|^2}\right|=\sup_{\xi\in\RR^3}|\xi|^{a-2}e^{-t|\xi|^2}\leq Ct^{-\frac{a-2}{2}}\quad\text{for all}\quad t>0.
\end{equation*}

To deal with the second term on the right-hand side of \eqref{eq.omega}, we fix $T>0$ and consider $s,t\in[0,T]$.
Thus, by the H\"older continuity of $\gamma(t)$, there is a constant $C=C(T)>0$ such that $\big|\gamma(t)-\gamma(s)\big|\leq C(T)|t-s|^\alpha$. Hence
\begin{equation}\label{eq.om:a}
\begin{split}
\left|\int_0^te^{-(t-s)|\xi|^2}\left(e^{i\gamma(s)\cdot\xi}-e^{i\gamma(t)\cdot\xi}\right)\;\mathrm{d}s\right|
\leq&C(T)\int_0^te^{-(t-s)|\xi|^2}|\xi||t-s|^\alpha\;\mathrm{d}s\\
\leq&C(T)\frac{1}{|\xi|^{1+2\alpha}}\int_0^\infty e^{-s}s^\alpha\;\mathrm{d}s.
\end{split}
\end{equation}
Consequently, using estimate \eqref{eq.om:2} for $|\xi|\leq1$ and estimate \eqref{eq.om:a} for $|\xi|\geq1$, we obtain for each $t\in[0,T]$
\begin{equation*}
\begin{split}
\text{ess}\sup_{\xi\in\RR^3}|\xi|^a\big|\widehat{\omega_0}(t,\xi)\big|\leq&
\text{ess}\sup_{|\xi|\leq1}|\xi|^a\big|\widehat{\omega_0}(t,\xi)\big|+\text{ess}\sup_{|\xi|\geq1}|\xi|^a\big|\widehat{\omega_0}(t,\xi)\big|\\
\leq& 2\sup_{|\xi|\leq1}|\xi|^{a-2}+C(T)\sup_{|\xi|\geq1}|\xi|^{a-1-2\alpha}.
\end{split}
\end{equation*}
Here, the right-hand side is finite for every $a\in[2,1+2\alpha]$, and this interval is non-empty if $\alpha>\frac12$. Thus, we have proved that $\omega_0\in \mathcal{Y}_T^a$ for each $a\in[2,1+2\alpha]$ and $T>0$.

\indent {\it Step 4.}
By Lemma \ref{lem:interpol},  we have
\begin{equation}\label{eq.wlq}
\|\omega_0(t)\|_{L^q}\leq C\|\omega_0(t)\|_{\mathcal{PM}^2}^{1-\beta}\|\omega_0(t)\|_{\mathcal{PM}^a}^{\beta},
\end{equation}
for each $q\in(3,\frac{3}{3-a})$ and $\beta=\frac{1}{a-2}(1-\frac3q)$.
Since $\omega_0\in \mathcal{Y}_T^a$ with arbitrary $a\in[2,1+2\alpha]$, we obtain from inequality \eqref{eq.wlq}, the following estimate $ \|\omega_0(\cdot,t)\|_{L^q(\RR^3)}\leq Ct^{\frac12\left(\frac3q-1\right)}$  for every $q\in(3,\frac{3}{2-2\alpha})$, a constant $C=C(q,\alpha,T)>0$ and all $t\in(0,T]$.

\textit{Step 5.} It remains for us to prove that $u=u(x,t)$ is a $C^\infty$-solution of the heat equation~\eqref{eq.heat} on $(\RR^3\times\RR^+)\backslash\Gamma$.
Let $\varphi\in C^\infty_{\rm c}\big((\RR^3\times\RR^+)\backslash\Gamma\big)$ be an arbitrary test function. We multiply equations \eqref{eq.FHS} by $\widehat{\varphi}(\xi,t)$ and integrate on $\RR^3\times\RR^+$  to obtain (after integration by parts with respect to $t$) the equation
\begin{equation*}
-\int_0^\infty\int_{\RR^3}\widehat{u}(\xi,t)\widehat{\varphi}_t(\xi,t)\;\mathrm{d}\xi\mathrm{d}t=
-\int_0^\infty\int_{\RR^3}\widehat{u}(\xi)|\xi|^2\widehat{\varphi}(\xi,t)\;\mathrm{d}\xi\mathrm{d}t+\int_0^\infty\int_{\RR^3}e^{i\gamma(t)\cdot\xi}\widehat{\varphi}(\xi,t)\;\mathrm{d}\xi\mathrm{d}t.
\end{equation*}
Since $\mathrm{supp}\,\varphi\subset(\RR^3\times\RR^+)\backslash\Gamma$, by properties of the Fourier transform, we obtain
\begin{equation}\label{eq.11a}
\int_0^\infty\int_{\RR^3}e^{i\gamma(t)\cdot\xi}\widehat{\varphi}(\xi,t)\;\mathrm{d}\xi\mathrm{d}t=\int_0^\infty \varphi(\xi(t),t)\,\mathrm{d}t=0.
\end{equation}
Hence, by the Plancherel formula, we have
\begin{equation}\label{eq.p17}
-\int_0^\infty\int_{\RR^3}u(x,t)\varphi_t(x,t)\;\mathrm{d}x\mathrm{d}t=-\int_0^\infty\int_{\RR^3}u(x,t)\Delta\varphi(x,t)\;\mathrm{d}x\mathrm{d}t,
\end{equation}
so, the function $u=u(x,t)$ is a distributional solution of heat equation \eqref{eq.heat} in $(\RR^3\times\RR^+)\backslash\Gamma$. It follows from the formula $u(x,t)=\omega_0(x,t)+(4\pi|x-\gamma(t)|)^{-1}$ and from \eqref{eq.wlq} that $u\in L^1_{\rm loc}(\RR^3\times(0,\infty))$, hence, the function
$u=u(x,t)$ is smooth on $(\RR^3\times\RR^+)\backslash\Gamma$ by the Weyl theorem in \cite{RS75}.
\end{proof}
\section{Singular solutions to the Navier-Stokes system}\label{PROF}
\setcounter{section}{5}\setcounter{equation}{0}
In this section, we construct a vector field $u=(u_1,u_2,u_3)$ and a pressure $p=p(x,t)$ with properties stated in Theorem \ref{THM.SNS} as a solution to the singular initial value problem \eqref{eq.SNS}. First, we recall from \cite{CK04} a result on the existence of  a family of tempered distributions $u=u(t)$ which satisfies problem \eqref{eq.SNS}.

\begin{theorem}\label{THM.1}
Let $\gamma(t):[0,\infty)\to \RR^3$ be arbitrary. Assume that $|\kappa|<\frac{1}{8\eta_2}$ where $\eta_2$ is constant from inequality \eqref{eq.Ya-Est} with $a=2$. Then, the singular initial value problem \eqref{eq.SNS} has a solution $u\in \mathcal{Y}_\infty^2=C_{\rm w}\big([0,\infty);\mathcal{PM}^2\big)$ in the sense of Definition \ref{def1}. This is a unique solution satisfying $|||u|||_{2,\infty}\leq4|\kappa|$.
\end{theorem}
\begin{proof}
This theorem is a particular case of \cite[Theorem 4.1]{CK04}, however, we  sketch its proof for the completeness of the exposition. We write equation \eqref{FDuh} in the following form
\begin{equation}\label{eq.uBy}
u=B(u,u)+y,
\end{equation}
where (see Definition \ref{def1}) the bilinear form $B(\cdot,\cdot)$ is defined by its Fourier transform
\begin{equation}\label{eq.Bform}
\widehat{B(u,v)}(\xi,t)=\int_0^te^{-(t-\tau)|\xi|^2}i\xi\widehat{\mathbb{P}}(\xi)\widehat{(u\otimes v)}(\xi,\tau)\;\mathrm{d}\tau
\end{equation}
and
$$\widehat{y}=\kappa\int_0^te^{-(t-\tau)|\xi|^2}\widehat{\mathbb{P}}(\xi)e^{i\gamma(\tau)\cdot\xi}\;\mathrm{d}\tau.$$
Since $\sup_{\xi\in\RR^3\backslash\{0\}}\big|\widehat{\mathbb{P}(\xi)}\big|\leq2$,
following the arguments of Step 1 in the proof of Theorem~\ref{THM.HEAT}, we obtain immediately that $y\in C_{\rm w}\big([0,\infty),\mathcal{PM}^2\big)$ and
$|||y|||_{2,\infty}\leq2|\kappa|$.
Next, estimate \eqref{eq.Ya-Est} with $a=2$ and $T=\infty$ guarantees that the bilinear form $B(\cdot,\cdot)$ satisfies the inequality
\begin{equation*}
|||B(u,v)|||_{2,\infty}\leq \eta_2|||u|||_{2,\infty}|||v|||_{2,\infty}
\end{equation*}
for all $u,\,v\in \mathcal{Y}_\infty^2$. Hence, by Lemma \ref{lem:xyB} with $L=0$ and $\lambda=0$, equations \eqref{eq.uBy} have a unique solution satisfying $|||y|||_{2,\infty}\leq4|\kappa|$.
\end{proof}
\begin{corollary}\label{coro.5.2}
The time dependent family of tempered distributions $u(t)$ from Theorem~\ref{THM.1} is represented by a locally integrable function $u\in C_{\rm w}([0,\infty),L^{3,\infty}(\RR^3))$.
\end{corollary}
\begin{proof}
This is an immediate consequence of the continuous imbedding $\mathcal{PM}^2\subset L^{3,\infty}(\RR^3)$ proved in Lemma \ref{lem.weakl-3}.
\end{proof}
In the next step, we compare $u(x,t)$ with the shifted Slezkin-Landau solution $V^c_{\gamma}(x,t)=V^c(x-\gamma(t))$. We begin with an integral representation of $V^c$ analogous to that one for $\widehat{u}(\xi,t)$ in \eqref{FDuh}.
\begin{lemma}
Let $(V^c,Q^c)$ be the Slezkin-Landau solution given by  formula \eqref{sing-sol}. Denote $V_{\gamma}^c(x,t)\equiv V^c(x-\gamma(t))$. Then
\begin{equation}\label{eq.V}
\widehat{V^c_\gamma}(\xi,t)=\widehat{B(V^c_\gamma,V^c_{\gamma})}(\xi,t)+\kappa\int_0^te^{-(t-\tau)|\xi|^2}\mathcal{\mathbb{P}}(\xi)e^{i\gamma(\tau)\cdot\xi}\;\mathrm{d}\tau\bar{e}_1
+\widehat{y_0}(\xi,t),
\end{equation}
for all $(\xi,t)\in\RR^3\times\RR^+$, where
\begin{equation}\label{eq.y0}
\widehat{y_0}(\xi,t)
=-|\xi|^2\int_0^te^{-(t-\tau)|\xi|^2}\left(e^{i\gamma(\tau)\cdot\xi}-e^{i\gamma(t)\cdot\xi}\right)\widehat{V}^c(\xi)\;\mathrm{d}\tau+e^{-t|\xi|^2}
e^{i\gamma(t)\cdot\xi}\widehat{V}^c(\xi),
\end{equation}
and the bilinear form $B(\cdot,\cdot)$ is defined in \eqref{eq.Bform}.
\end{lemma}
\begin{proof}
Since $V^c(x)$ is a distribution solution of the Navier-Stokes system with the singular force $\kappa\delta_0\bar{e}_1$, in the Fourier variables, we have
\begin{equation}\label{eq.FL}
|\xi|^2\widehat{V}^c(\xi)=-\widehat{\mathbb{P}}(\xi)i\xi\widehat{\left(V^c\otimes V^c\right)}(\xi)+\kappa\widehat{\mathbb{P}}(\xi)\bar{e}_1.
\end{equation}
To derive the Fourier integral representation \eqref{eq.V}, we notice that $e^{i\gamma(t)\cdot\xi}\widehat{V}^c(\xi)=\widehat{V^c_\gamma}(\xi,t)$
and $e^{i\gamma(t)\cdot\xi}\widehat{(V^c\otimes V^c)}(\xi)=\widehat{(V^c_\gamma\otimes V^c_\gamma)}(\xi)$. Hence, multiplying equation \eqref{eq.FL} by $e^{i\gamma(t)\cdot\xi}$, we obtain the relation
\begin{equation*}\label{eq.FLG}
|\xi|^2\widehat{V^c_\gamma}(\xi,t)=-\widehat{\mathbb{P}}(\xi)i\xi\widehat{(V^c_\gamma\otimes V^c_{\gamma})}(\xi,t)+\kappa e^{i\gamma(t)\cdot\xi}\bar{e}_1,
\end{equation*}
for all $\xi\in\RR^3$ and $t\geq0$, which is equivalent by a direct calculation to the following formula
\begin{equation*}
\begin{split}
\widehat{V^c_\gamma}(\xi,t)\equiv&-\int_0^te^{-(t-\tau)|\xi|^2}\mathcal{\mathbb{P}}(\xi)i\xi\widehat{(V^c_\gamma\otimes V^c_\gamma)}(\xi,\tau)\;\mathrm{d}\tau+\kappa\int_0^te^{-(t-\tau)|\xi|^2}\mathcal{\mathbb{P}}(\xi)e^{i\gamma(\tau)\cdot\xi}\;\,\mathrm{d}\tau\bar{e}_1\\
&-|\xi|^2\int_0^te^{-(t-\tau)|\xi|^2}\widehat{V^c_\gamma}(\xi,\tau)\;\mathrm{d}\tau+\widehat{V^c_{\gamma}}(\xi,t).
\end{split}
\end{equation*}
Let us modify the sum of the last two terms  on the right-hand side of the above identity using the relation $\widehat{V_\gamma^c}(\xi,t)=e^{i\gamma(t)\cdot\xi}\widehat{V^c}(\xi)$ in the following way
\begin{equation}\label{eq.y00}
\begin{split}
\widehat{y_0}(\xi,t)=&-|\xi|^2\int_0^te^{-(t-\tau)|\xi|^2}e^{i\gamma(\tau)\cdot\xi}\widehat{V^c}(\xi)\;\mathrm{d}\tau+\widehat{V_{\gamma}^c}(\xi,t)\\
=&-|\xi|^2\int_0^te^{-(t-\tau)|\xi|^2}\left(e^{i\gamma(\tau)\cdot\xi}-e^{i\gamma(t)\cdot\xi}\right)\widehat{V^c}(\xi)\;\mathrm{d}\tau+e^{-t|\xi|^2}
e^{i\gamma(t)\cdot\xi}\widehat{V^c}(\xi).
\end{split}
\end{equation}
Hence, recalling the bilinear form \eqref{eq.Bform},
we obtain the integral equation \eqref{eq.V}.
\end{proof}

Now, we consider the difference
\begin{equation*}
\widehat{\omega}(\xi,t)=\widehat{u}(\xi,t)-\widehat{V_\gamma^c}(\xi,t),
\end{equation*}
which by equations \eqref{FDuh} and \eqref{eq.V} satisfies
\begin{equation}\label{eq.DFw}
\begin{split}
\widehat{\omega}(\xi,t)=&\widehat{B(u,u)}(\xi,t)-\widehat{B(V^c_\gamma,V^c_\gamma)}(\xi,t)+\widehat{y_0}(\xi,t)\\
=&\widehat{B(\omega,\omega)}(\xi,t)+\widehat{B(V^c_\gamma,\omega)}(\xi,t)+\widehat{B(\omega,V^c_\gamma)}(\xi,t)+\widehat{y_0}(\xi,t).
\end{split}
\end{equation}
Following the notations from the proof of Theorem \ref{THM.1}, we write equation \eqref{eq.DFw} in the following abridged form
\begin{equation}\label{eq.w:B}
\omega=B(\omega,\omega)+B(V^c_\gamma,\omega)+B(\omega,V^c_\gamma)+y_0,
\end{equation}
where the bilinear form $B(\cdot,\cdot)$ is defined in \eqref{eq.Bform} and $y_0$ in \eqref{eq.y0}.

Let us prove a counterpart of Theorem \ref{THM.1} in the case of equation \eqref{eq.w:B}.
\begin{theorem}\label{THM.w}
Let $\gamma:[0,\infty)\to\RR^3$ be arbitrary. Then there exists $c_{\rm  0}>12K\eta_2$ where $\eta_2$ is constant from \eqref{eq.Ya-Est} and $K$ defined in \eqref{eq.Landau-small} such that for every $|c|\geq c_0$ equation \eqref{eq.w:B} has a solution $\omega\in C_{\rm w}\big([0,\infty);\mathcal{PM}^2\big)$. Moreover, this is a unique solution which satisfies $\sup_{t>0}\|\omega(t)\|_{\mathcal{PM}^2}\leq  \frac{1-2\eta_2\|V^c\|_{\mathcal{PM}^2}}{2\eta_2}$.
\end{theorem}
\begin{proof}
Here, the reasoning is completely analogous to the one in the proof of Theorem~\ref{THM.1}. We apply Lemma \ref{lem:xyB} to equation \eqref{eq.w:B} with the linear operator
$L\,\omega=B(V^c_\gamma,\omega)+B(\omega,V^c_\gamma).$
It follows from estimate \eqref{eq.Ya-Est} with $a=2$ and $T=\infty$ that
\begin{equation*}
\sup_{t>0}\|L\,\omega(t)\|_{\mathcal{PM}^2}\leq 2\eta_2\sup_{t>0}\|V^c_\gamma\|_{\mathcal{PM}^2}\sup_{t>0}\|\omega(t)\|_{\mathcal{PM}^2}=2\eta_2\|V^c\|_{\mathcal{PM}^2}\sup_{t>0}\|\omega(t)\|_{\mathcal{PM}^2}.
\end{equation*}

Now, we apply Lemma \ref{lem:xyB} with $\lambda=2\eta_2\|V^c\|_{\mathcal{PM}^2}$ and $\eta=\eta_2$ in the following way. Notice that, by Lemma~\ref{lem.Landau-small}, we have $\lambda\leq2\eta_2\frac{K}{|c|}<\frac16$ for all $|c|\geq c_0>12K\eta_2$.
Next, by a direct calculation, if $c_0>12K\eta_2$, we have
\begin{equation*}
\|y_0(t)\|_{\mathcal{PM}^2}\leq\|V^c\|_{\mathcal{PM}^2}\sup_{\xi\in\RR^3}\big(1-e^{-t|\xi|^2}\big)+\|V^c\|_{\mathcal{PM}^2}\leq2\|V^c\|_{\mathcal{PM}^2}.
\end{equation*}
Thus, by a direct calculation, if $c_0>12K\eta_2$ we have
\begin{equation}\label{eq.y0*}
4\eta_2\sup_{t>0}\|y_0(t)\|_{\mathcal{PM}^2}\leq8\eta_2\|V^c\|_{\mathcal{PM}^2}<(1-\lambda)^2
\end{equation} for all $|c|\geq c_0$, which is possible due to Lemma \ref{lem.Landau-small}.

Finally, using inequality \eqref{eq.Ya-Est} with $a=2$ and $T=\infty$ and applying Lemma \ref{lem:xyB}, we obtain a solution $\omega\in \mathcal{Y}^2_\infty$ of equation \eqref{eq.DFw}.
\end{proof}
The following corollary is a direct consequence of the uniqueness of solution to the considered equations.
\begin{corollary}
 Let $u$ be a solution to problem \eqref{eq.SNS} constructed in Theorem \ref{THM.1}. Then, choosing $c_0$ sufficiently large, we obtain that the solution $\omega$ in Theorem \ref{THM.w} is of the form $\omega=u-V^c_\gamma$ for all $|c|\geq c_0$.
\end{corollary}
\begin{proof}
We use the notation from the proof of Theorem \ref{THM.w}. By Lemma \ref{lem:xyB}, estimate \eqref{eq.y0*}, and inequality \eqref{eq.Landau-small}, the solution $\omega$ of equation \eqref{eq.DFw} satisfies
\begin{equation*}
\|\omega\|_{\mathcal{Y}_\infty^2}\leq\frac{4\|V^c\|_{\mathcal{PM}^2}}{1-2\eta_2\|V^c\|_{\mathcal{PM}^2}}\leq\frac{4K}{|c|-2\eta_2K}.
\end{equation*}
Thus, for sufficiently large $|c|$, we have $\omega=u-V^c_\gamma$ by the uniqueness of the solution $u$ to equation \eqref{eq.uBy} established in Theorem \ref{THM.1} and the uniqueness of $\omega$ from Theorem~\ref{THM.w}.
\end{proof}

In the next step, we prove that the solution $\omega\in\mathcal{Y}^2_\infty$  of equation \eqref{eq.w:B} satisfies $\omega\in\mathcal{Y}_{T}^a$ with suitable $a>2$.
\begin{theorem}\label{THM.w.Ya}
Let $\gamma:[0,\infty)\to\RR^3$ be H\"older continuous with an exponent $\alpha\in(\frac12,1]$. Assume, in addition, that $|c|>\max\{c_{\rm 0},\,2(8\eta_a+\eta_2)K\}$. Then, the solution $\omega$ constructed in Theorem \ref{THM.w} satisfies $\omega\in \mathcal{Y}^a_T$ for each $a\in[2,1+2\alpha]$ and each $T>0$.
\end{theorem}
In the proof of this theorem, we need the following property of $y_0$ defined by formula~\eqref{eq.y0}.
\begin{lemma}\label{lem.y0-a}
Let $\gamma:[0,\infty)\rightarrow\RR^3$ be H\"older continuous with an exponent $\alpha\in(\frac12,1]$. Then, for each $T>0$ and each $a\in[2,1+2\alpha]$, we have $y_0\in\mathcal{Y}_{T}^a$.
\end{lemma}
\begin{proof}
First, we notice that $y_0(t)$ given by \eqref{eq.y0} has the form of the tempered distribution $\omega(t)$ defined in the case of the heat equation by formula \eqref{eq.omega} with $\frac{1}{|\xi|^2}$ replaced by $\widehat{V^c}(\xi)$.
Since $|\widehat{V^c}(\xi)|\leq \frac{C}{|\xi|^2}$, to complete the proof of this lemma, it suffices to repeat the reasoning form Step 3 of the proof of Theorem \ref{THM.HEAT}.
\end{proof}
\begin{proof}[Proof of Theorem \ref{THM.w.Ya}]
The solution $\omega\in \mathcal{Y}^2_\infty=C_{\rm w}\big([0,\infty);\mathcal{PM}^2\big)$ is obtained in Theorem~\ref{THM.w} as a limit of the sequence $\{\omega_n\}\subset\mathcal{Y}^2_\infty$  defined by the recurrence formula
\begin{equation}\label{eq.recurr}
\omega_1\equiv0,\qquad\omega_{n+1}=\mathcal{T}(\omega_n)=B(\omega_n,\omega_n)+B(V^c_\gamma,\omega_n)+B(\omega_n,V^c_\gamma)+y_0.
\end{equation}
Using Lemma \ref{lem.y0-a} and estimate \eqref{eq.Ya-Est},
we obtain immediately that $\omega_n\in \mathcal{Y}^a_T$ for each $n\geq0$ and $T>0$.

We define the set $B(0,A)=\{\,\omega\in \mathcal{Y}^a_T:\,|||\omega|||_{2,\infty}\leq A\,\}$ which is a closed subset of $\mathcal{Y}^a_T$. 
Notice that $B(0,A)$ is bigger than  the  ball of radius $A$
in $\mathcal{Y}^a_T$, because this set  is defined via $|||\cdot|||_{2,\infty}$ which is not the full norm in $\mathcal{Y}_T^a$, see \eqref{Ya}. To show that the sequence $\{\omega_n\}$ converges in $\mathcal{Y}_T^a$, it suffices to prove that the mapping $\mathcal{T}$ defined in \eqref{eq.recurr} satisfies $\mathcal{T}:\, B(0,A)\to B(0,A)$ with $A=\frac{4\|V^c\|_{\mathcal{PM}^2}}{1-2\eta_2\|V^c\|_{\mathcal{PM}^2}}$ and is a contraction.

For $\omega_n\in B(0,A)$, by a simple calculation involving \eqref{eq.recurr} and inequality \eqref{eq.Ya-Est} with $a=2$, we get
\begin{equation*}
\begin{split}
\big\|\mathcal{T}(\omega_{n})\big\|_{\mathcal{PM}^2}\leq & \eta_2\|\omega_n\|_{\mathcal{PM}^2}^2+2\eta_2\|V^c\|_{\mathcal{PM}^2}\|\omega_n\|_{\mathcal{PM}^2}+\|y_0\|_{\mathcal{PM}^2}\\
\leq & \eta_2A^2+2\eta_2\|V^c\|_{\mathcal{PM}^2}A+2\|V^c\|_{\mathcal{PM}^2}\\
=&\frac{2\|V^c\|_{\mathcal{PM}^2}}{1-2\eta_2\|V^c\|_{\mathcal{PM}^2}}\left(1+2\eta_{2}\|V^c\|_{\mathcal{PM}^2}
+\frac{8\eta_2\|V^c\|_{\mathcal{PM}^2}}{1-2\eta_2\|V^c\|_{\mathcal{PM}^2}}\right)\\
\leq&\frac{4\|V^c\|_{\mathcal{PM}^2}}{1-2\eta_2\|V^c\|_{\mathcal{PM}^2}},
\end{split}
\end{equation*}
because $1+2\eta_{2}\|V^c\|_{\mathcal{PM}^2}
+\frac{8\eta_2\|V^c\|_{\mathcal{PM}^2}}{1-2\eta_2\|V^c\|_{\mathcal{PM}^2}}\leq2$ by calculations which are similar to those in~\eqref{eq.y0*}.
Hence, $\mathcal{T}:\,B(0,A)\rightarrow B(0,A)$.

Next, let $\omega,\,\bar{\omega}\in\,B(0,A)\cap \mathcal{Y}^a_T$. Using estimate \eqref{eq.Ya-Est}, we have
\begin{equation*}
\begin{split}
&|||\mathcal{T}\omega-\mathcal{T}\bar{\omega}|||_{a,T}\\
\leq& |||L\,(\omega-\bar{\omega})|||_{a,T}+|||B(\omega,\omega-\bar{\omega})|||_{a,T}+|||B(\omega-\bar{\omega},\bar{\omega})|||_{a,T}\\
\leq&2\eta_a\|V^c\|_{\mathcal{PM}^2}|||\omega-\bar{\omega}|||_{a,T}+\eta_a|||\omega|||_{2,\infty}|||\omega-\bar{\omega}|||_{a,T}+
\eta_a|||\bar{\omega}|||_{2,\infty}|||\omega-\bar{\omega}|||_{a,T}\\
\leq&\big(2\eta_a\|V^c\|_{\mathcal{PM}^2}+\eta_a(|||\omega|||_{2,\infty}+|||\bar{\omega}|||_{2,\infty})\big)|||\omega-\bar{\omega}|||_{a,T}.
\end{split}
\end{equation*}
Since $c> 2(8\eta_a+\eta_2)K$, we have $\|V^c\|_{\mathcal{PM}^2}<\frac{1}{4\eta_a}$ by Lemma \ref{lem.Landau-small}. Moreover, it follows from Theorem \ref{THM.w} that $\sup_{t>0}\|\omega(t)\|_{\mathcal{PM}^2}\leq  \frac{4K}{c-2\eta_2K}<\frac{1}{4\eta_a}$.
Thus, $\mathcal{T}$ is a contraction in the norm $|||\cdot|||_{a,T}$ for sufficiently large $c$ which implies that the sequence $\{\omega_n\}$ converges toward $\omega\in\mathcal{Y}_T^a$.
\end{proof}
We are in a position to complete the proof of the main result from this work.
\begin{proof}[Proof of Theorem \ref{THM.SNS}]
\textit{Step 1: Existence of $u$ and $p$.} Choosing $|c|$ large enough, by Proposition \ref{prop:sing-sol}, we obtain $|\kappa(c)|<\frac{1}{8\eta_2}$.
Hence, Theorem \ref{THM.1} provides a solution $u\in C_{\rm w}\big([0,\infty);\mathcal{PM}^2\big)$. In fact, by Corollary~\ref{coro.5.2}, we have
$u\in L^\infty\big([0,\infty);L^{3,\infty}(\RR^3)\big)$. Moreover, Lemma \ref{lem.bil-weakl3} implies
\begin{equation}\label{eq.u2}
u\otimes u\in L^\infty\big([0,\infty);L^{\frac32,\infty}(\RR^3)\big).\end{equation}

Now, we determine the pressure from equation \eqref{eq.SNS} in the usual way computing the divergence  in the sense of $\mathcal{S}'(\RR^3)$ of the equation in \eqref{eq.SNS}. Thus, by a direct calculation, we obtain in the Fourier variables
\begin{equation}\label{eq.pressure***}
\widehat{p}(\xi,t)=\sum_{i,\,j=1}^3\frac{\xi_i\xi_j}{|\xi|^2}\widehat{(u\otimes u)}(\xi,t)-\kappa\frac{i\xi\cdot\bar{e}_1}{|\xi|^2}e^{i\gamma(t)\cdot\xi}
\end{equation}
or, equivalently,
\begin{equation}\label{eq.presure}
p(x,t)=\sum_{i,\,j=1}^3R_iR_j(u_iu_j)(x,t)-\kappa C_3\frac{x_1-\gamma(t)}{|x-\gamma(t)|^3},
\end{equation}
with a suitable explicit constant $C_3\in\RR$ and the Riesz transforms $R_j$ which are bounded in $L^{p,\infty}(\RR^3)$ for each $p\in(1,\infty)$. Since $\frac{x_1}{|x|^3}\in L^{\frac32,\infty}(\RR^3)$, by \eqref{eq.u2}, we obtain
\begin{equation}\label{eq.pre-est}
p\in L^\infty\big([0,\infty);L^{\frac32,\infty}(\RR^3)\big).
\end{equation}
In particular,   both $u(x,t)$ and $p(x,t)$ are locally integrable functions and not just tempered distributions.

\textit{Step 2.} Now, we prove that $(u,p)$  is a distributional solution of the Navier-Stokes system \eqref{eq.NS} (\textit{i.e.} this system without an external force) in $(\RR^3\times\RR^+)\backslash\Gamma$.
Since $u(x,t)$ satisfies \eqref{FDuh}, we have
\begin{equation*}
0=\int_{\RR^3}i\xi\cdot\widehat{u}(\xi,t)\widehat{\psi}(\xi)\,\mathrm{d}\xi=\int_{\RR^3}u(x,t)\cdot\nabla\psi(x)\;\mathrm{d}x \quad \text{for every}\quad \psi \in C_{\rm c}^\infty\big((\RR^3\times\RR^+)\backslash\Gamma\big).
\end{equation*}
Moreover, it follows from integral equation \eqref{FDuh} that
\begin{equation*}
\partial_t\widehat{u}(\xi,t)+|\xi|^2\widehat{u}(\xi,t)+\widehat{\mathbb{P}}(\xi)i\xi\cdot\widehat{(u\otimes u)}(\xi,t)=\kappa\widehat{\mathbb{P}}(\xi)e^{i\gamma(t)\cdot\xi}\bar{e}_1.
\end{equation*}
Using the definition of $\mathbb{P}$ in \eqref{eq.Leray-pro} and equation \eqref{eq.pressure***}, we obtain
\begin{equation}\label{eq.mild-t}
\partial_t\widehat{u}(\xi,t)+|\xi|^2\widehat{u}(\xi,t)+i\xi\cdot\widehat{u\otimes u}(\xi,t)+i\xi\widehat{p}(\xi,t)=\kappa e^{i\gamma(t)\cdot\xi}\bar{e}_1.
\end{equation}
We multiply both sides of \eqref{eq.mild-t} by $\widehat{\varphi}(\xi,t)$, where $\varphi=(\varphi_1,\varphi_2,\varphi_3)\in \big(C^\infty_{\rm c}\big((\RR^3\times\RR^+)\backslash\Gamma\big)\big)^3$ is arbitrary test function. Integrating the resulting equation with respect to space and time, using the Plancherel theorem and repeating the argument from \eqref{eq.11a}-\eqref{eq.p17}, we obtain
\begin{equation*}
\int_0^t\int_{\RR^3}\Big(u\cdot\big(-\partial_t\varphi-\Delta\varphi\big)-\sum_{i,j}u_iu_j\partial_{x_i}\varphi_j-p\sum_{i}\partial_{x_i}\varphi_i\Big)\;\mathrm{d}x\mathrm{d}t=0.
\end{equation*}
Hence, $(u,p)$ satisfies system \eqref{eq.NS} in $(\RR^3\times\RR^+)\backslash\Gamma$ in the sense of distribution.

\textit{Step 3. Regularity of $u$ and $p$.}
First,
for $\alpha\in\big(\frac12,1\big]$, we apply a regularity criterion for distributional solutions of the Navier-Stokes equations which was recently proved in \cite{LT2013}.
Let $B(x_0,r_0)\times[t_0,t_1]\subset(\RR^3\times\RR^+)\backslash\Gamma$ be an arbitrary cylinder. It follows immediately from \eqref{eq.pre-est} that
the pressure $p\in L^\infty\big([t_1,t_2];L^1(B(x_0,r_0))\big)$. Moreover,  using  Theorem \ref{THM.SNS} and the estimate
 $\|u\|_{L^{3,\infty}(\RR^3)}\leq C\|u\|_{\mathcal{PM}^2}$ from Lemma \ref{lem.weakl-3}, we obtain
$
 \|u\|_{L^\infty\left([t_0,t_1];L^{3,\infty}(B(x_0,r_0))\right)}\leq 4C|\kappa|.
$
We choose $|\kappa|$ in Theorem \ref{THM.SNS} sufficiently small such that the constant $4C|\kappa|$ meets requirements in \cite[Theorem 1.1]{LT2013}, which imply $u\in L^\infty(B(x_0,r_0)\times[t_1+\delta,t_2])$, where $\delta>0$ is a small and fixed number. Thus, we get $u\in L^\infty_{\rm loc}\big((\RR^3\times\RR^+)\backslash\Gamma\big)$.

Next, we consider $\alpha\in(\frac34,1]$.
By Theorem \ref{THM.w.Ya}, we have
$\omega\in \mathcal{Y}^a$ for every $a\in[2,1+2\alpha]$ where $1+2\alpha>\frac52$.
Hence,  we may apply  Lemma \ref{lem.imbed} which implies that $\nabla \omega \in L^2_{\rm loc}(\RR^3\times\RR^+)$ and, in a consequence, we
have
 $\nabla u\in L^2_{\rm loc}\big((\RR^3\times\RR^+)\backslash\Gamma\big)$. Thus, using the classical local regularity criterion by Serrin~\cite{SER}
 we obtain that $u=u(x,t)$ is a smooth solution of the Navier-Stokes system~\eqref{eq.NS} outside the curve $\Gamma$, \textit{i.e.} $u\in C^\infty\big((\RR^3\times\RR^+)\backslash\Gamma\big)$.

 Next, we show the regularity of the pressure $p=p(x,t)$. Since
 \begin{equation*}
(V^c_\gamma\cdot\nabla)V^c_\gamma+\nabla Q^c_\gamma-\Delta V^c_\gamma=\kappa(c)\delta_{\gamma(t)}\bar{e}_1
 \end{equation*}
 in the sense of distribution, subtracting this equation from the equation for $u(x,t)$ in \eqref{eq.SNS} and using the divergence free condition, we obtain
 \begin{equation}\label{eq.pre-diff}
 \begin{split}
 -\Delta(p-Q^c_{\gamma})=&\Div \left((u\cdot\nabla) u-(V^c_\gamma\cdot\nabla) V^c_\gamma\right)\\
 =&\left(\nabla u\otimes\nabla u-\nabla V_{\gamma}^c\otimes\nabla V_{\gamma}^c\right)
 \end{split}
 \end{equation}
in the sense of distribution over $\RR^3\times\RR^+$. Since $p-Q^c_\gamma\in L^1_{\rm loc}(\RR^3\times\RR^+)$ and since
$\nabla u\otimes\nabla u-\nabla V_{\gamma}^c\otimes\nabla V_{\gamma}^c\in C^\infty\big((\RR^3\times\RR^+)\backslash\Gamma\big)$, we obtain $p-Q^c_\gamma\in C^\infty\big((\RR^3\times\RR^+)\backslash\Gamma\big)$ by  the Weyl regularity theorem (see {\it e.g.} \cite[Theorem IX.25]{RS75}).

 \textit{Step 4. Singularity of $u$ and $p$ at the curve.}
Now, we prove properties of $u$ and $p$  in Theorem \ref{THM.SNS}.\eqref{eq.thm-item-2}.
Since $\omega\in\mathcal{Y}_T^a$ for each $a\in[2,1+2\alpha]$, using the inequality in Lemma~\ref{lem:interpol} analogously as in Step 4 of the proof of Theorem \ref{THM.HEAT}, we obtain that
 $\omega(\cdot,t)=u(\cdot,t)-V^c(\cdot-\gamma(t))\in L^q(\RR^3)$ for each $q\in(3,\frac{3}{2-2\alpha})$ and all $t\in(0,T]$ with arbitrary $T>0$.

 To show the integrability of $p-Q^c_\gamma$ stated in \eqref{eq.thm-item-2} of Theorem \ref{THM.SNS}, we rewrite equation~\eqref{eq.pre-diff} in the following form
 \begin{equation}\label{eq.pressure-dif}
  -\Delta(p-Q^c_{\gamma})=\Div \left((\omega\cdot\nabla) \omega+(V^c_\gamma\cdot\nabla) \omega+(\omega\cdot\nabla) V^c_\gamma\right),
 \end{equation}
 where $V^c_\gamma\in L^\infty([0,\infty);\mathcal{PM}^2)$ and $\omega=u-V_\gamma^c\in \mathcal{Y}_T^{1+2\alpha}$. Now, by direct calculations applied to formula \eqref{eq.pressure-dif} involving properties of the Fourier transform and Lemma \ref{lem.product}, we obtain
$p(\cdot,t)-Q_\gamma^c(\cdot,t)\in L^\infty_{\rm loc}((0,\infty);\mathcal{PM}^1\cap\mathcal{PM}^2)$. Thus, by Lemma \ref{lem:interpol}, we get $p(\cdot,t)-Q^c_\gamma(\cdot,t)\in L^q(\RR^3)$ for all
$q\in\big[2,\frac{3}{3-2\alpha}\big)$. Since, moreover, $V^c_\gamma\in L^\infty\big([0,\infty);L^{\frac32,\infty}(\RR^3)\big)$, we have the imbedding $p-Q^c_\gamma\in L^q(\RR^3)$ for all $q\in\big(\frac32,\frac{3}{3-2\alpha}\big)$.
\end{proof}

\section*{Acknowledgments}
The authors are greatly indebted to Eiji Yanagida for sending them his works
on singular solutions to parabolic equations.
This work was partially supported by  the Foundation for Polish Science operated within the Innovative Economy Operational Programme 2007-2013 funded by European Regional Development Fund (Ph.D. Programme: Mathematical
Methods in Natural Sciences) and  the MNiSzW grant No.~N~N201 418839.

\end{document}